\documentclass[12pt,a4paper]{article} 

\usepackage{amsmath,amssymb,amsthm,amsfonts,amscd,euscript,verbatim, t1enc, newlfont}
\usepackage{hyperref}
 \theoremstyle{plain}
\newtheorem{theo}{Theorem}[subsection]

\newtheorem{pr}[theo]{Proposition}

 \newtheorem{lem}[theo]{Lemma}

\theoremstyle{remark}
\newtheorem{rema}[theo]{Remark}

\theoremstyle{definition}
\newtheorem{defi}[theo]{Definition}
\newtheorem*{notat}{Notation}

 \newcommand\lan{\langle}
\newcommand\ra{\rangle}

\newcommand\ob{^{-1}}

\newcommand\dmge{DM^{eff}_{gm}{}}

\newcommand\dmgm{DM_{gm}}

\newcommand\dom{D^b(\mathcal{M}_1)}

\newcommand\hsing{H^{sing}}

\newcommand\mg{\mathcal{M}}
\newcommand\mgo{\mathcal{M}_1}

\newcommand\obj{Obj}

\newcommand\id{id}
\newcommand\cu{\underline{C}}

\newcommand\au{\underline{A}}
\newcommand\bu{\underline{B}}

\newcommand\aui{\underline{A}_i}

\newcommand\hrt{{\underline{Ht}}}
\newcommand\hw{{\underline{Hw}}}

\newcommand\z{{\mathbb{Z}}}

\newcommand\zop{{\mathbb{Z}[\frac{1}{p}]}}

\newcommand\ns{\{0\}}

\DeclareMathOperator\lalb{\operatorname{LAlb}}
\DeclareMathOperator\alb{\operatorname{Alb}}

\DeclareMathOperator\rpic{\operatorname{RPic}}

\newcommand\chow{Chow}

\newcommand\chowe{Chow^{eff}}

\newcommand\spe{\operatorname{Spec}\,}

\newcommand\mhmx{MHM(X)}
\newcommand\dbmhmx{D^bMHM(X)}

 \DeclareMathOperator\ke{\operatorname{Ker}}
 \DeclareMathOperator\cok{\operatorname{Coker}}
\DeclareMathOperator\imm{\operatorname{Im}}

\newcommand\wchow{{w_{Chow}}}

\newcommand\dbhp{D^b_{{\tilde H}_p}}
\newcommand\dmrp{D_{\mathcal{MR^P}}}

\begin{document}

%\special{papersize=21cm,29.7cm}

 \title{Weights and $t$-structures: in general triangulated categories, for $1$-motives,  mixed motives,  and
for mixed Hodge complexes and modules}
 \author{Mikhail V. Bondarko
\thanks{ %%!!!!
 The work is supported by RFBR
(grants no. 10-01-00287a and 11-01-00588a), by a Saint-Petersburg State University research grant no. 6.38.75.2011, and by the Federal Targeted Programme "Scientific and Scientific-Pedagogical Personnel of the Innovative Russia in 2009-2013" (Contract No. 2010-1.1-111-128-033); by Landau Network-Centro Volta and the Cariplo Foundation, and the University of Milano.}}
\maketitle
 
\begin{abstract}
We define and study {\it transversal} weight and $t$-structures (for triangulated categories); if a weight structure is transversal to a $t$-one, then it defines certain 'weights' for its heart. Our
results axiomatize and describe in detail  the relations between
the {\it Chow weight structure} $\wchow$ for Voevodsky's motives  (introduced in a preceding paper),  the
(conjectural) motivic $t$-structure, and
the conjectural weight filtration for them. This picture becomes
non-conjectural when restricted to the derived categories of
Deligne's $1$-motives (over a smooth base) and of Artin-Tate motives over number fields;  weight structures transversal to the canonical $t$-structures also exist for the Beilinson's $\dbhp$ (the derived category of graded
polarizable mixed Hodge complexes) and for the derived category of (Saito's) mixed Hodge modules. 

We also study weight filtrations for the heart of $t$ and (the degeneration of) weight spectral sequences. The corresponding relation between $t$ and $w$ is strictly weaker than transversality; yet it is easier to check, and we still obtain a certain filtration for (objects of) the heart of $t$ that is strictly respected by morphisms.

In a succeeding paper we apply the  results obtained in order to reduce the existence of Beilinson's mixed motivic sheaves (over a base scheme $S$) and 'weights' for them to (certain) standard motivic conjectures over a universal domain $K$.

% We study  conditions that ensure that $w$ induces a certain ('weight') filtration on the heart of $t$.

\end{abstract}

\tableofcontents

\section*{Introduction}

In this paper we study when a {\it weight structure} (as defined in \cite{bws}; in \cite{konk} weight structures were introduced independently under the name of co-$t$-structures) yields a certain 'weight filtration' for the heart of a $t$-structure in a triangulated category. We prove several formal results, and describe certain motivic and Hodge-theoretic examples of this situation.

The main reason to write this paper was to understand
the relation of the 'weight structure approach' to weights for motives (as introduced in \cite{bws}) with the 'classical' one.
Recall that  the triangulated category
$\dmgm$ of (geometric) Voevodsky's motives (over a perfect field $k$; all the motives that
we will consider in this paper will have  rational coefficients)
is widely believed to possess a certain motivic $t$-structure $t_{MM}$. Its heart
should be the category $MM$ of mixed motives, that should possess a
{\it weight filtration} whose factors yield certain semi-simple
abelian subcategories $MM_i\subset MM$ of pure motives of weight
$i$; the objects of $MM_i$ should be shifts of certain Chow
motives (note that in \cite{1} an embedding $\chow\to\dmgm$ was
constructed) by $[i]$. Since the existence of $t_{MM}$ is very
far from being known, people tried to find a candidate for the
weight filtration for $\dmgm$; this was to be a filtration by
triangulated subcategories that would restrict to the weight
filtration for $MM$. This activity was not really successful (in the  general case); this is no surprise since (for example) the weight
filtration for the motif of a smooth projective variety should
correspond to its {\it Chow-Kunneth decomposition}.

An alternative method for defining (certain) weights in $\dmgm$
was proposed and successfully implemented in \cite{bws}. To this
end {\it weight structures} were defined. This notion is a natural
important counterpart of $t$-structures; somewhat similarly to a
$t$-structure, a weight structure $w$ for a triangulated $\cu$ is
defined via certain $\cu^{w\le 0},\cu^{w\ge 0}\subset \obj\cu$.
The Chow weight structure  $\wchow$ (defined in \S6 of ibid.) % for motives over a field)
certainly does not yield a weight filtration for $\dmgm$ in the
sense described above (since $\dmgm^{\wchow\le 0}$ and $\dmgm^{\wchow\ge 0}$ are not stable with respect to shifts). Yet it allows us to define certain
(Chow)-weight filtrations and (Chow)-weight spectral sequences for
any cohomology of motives; for singular and \'etale cohomology
those are isomorphic to the 'classical' ones (that should also
have an expression in terms of the weight filtration for $\dmgm$), and this should
also be true for the 'mixed motivic' cohomology given by $t_{MM}$
(see Remark 2.4.3 of ibid. and \S\ref{sconj} below). Also note
here: the Chow weight structure for  Voevodsky's motives over a base scheme $S$
(introduced in \cite{hebpo} and \cite{brelmot}), is closely
related with the weights for  mixed complexes of sheaves introduced in
\S5.1.8 of \cite{bbd} (see \S\S3.4--3.6 of \cite{brelmot} for more detail), and with weights of mixed Hodge complexes (see \S\ref{shodge} below; we prove a very precise statement of this sort in the case when $S$ is the spectrum of a field $k\subset \mathbb{C}$). 

In the current paper we axiomatize and describe in
detail  the (conjectural) relations between $\wchow$, $t_{MM}$,
and the weight filtration for $\dmgm$ (we consider the latter in  Remark \ref{rwfilc} below).  To this end we introduce
the notion of {\it transversal} weight and $t$-structures. It is
no surprise that this notion has several non-conjectural (and
important) examples; this includes  the derived categories of Deligne's
$1$-motives (over a smooth base) and of Artin-Tate motives over number fields,  the derived category of (Saito's) mixed Hodge modules, and the Beilinson's  derived
category of graded polarizable mixed Hodge complexes. Certain results
of \cite{bvk} were very useful for studying these examples. %; cf. also \S3 of \cite{wildat}.

We prove several equivalent conditions for existence of
transversal weight and $t$-structures for a triangulated $\cu$.
One of them is the existence of a {\it strongly semi-orthogonal generating}
system of semi-simple abelian subcategories $\aui\subset \cu$ ($\aui$ are the factors of the 'weight filtration' of the heart of $t$). We
prove that any object of $\hw$ (the heart of $w$) splits into a sum of objects of $\aui[-i]$
(this should be a generalization of the Chow-Kunneth decomposition of motives of smooth projective varieties). This is a strong restriction on $w$; it demonstrates that the notion of transversal structures is
 quite distinct from the notion of {\it adjacent} weight and $t$-structures (introduced
in \S4.4 of \cite{bws}).

We 
also recall the notion of {\it weight spectral sequences} $T(H,-)$ for a (co)homological functor $H$, and prove: %that 
if those degenerate at $E_2$, then they induce a filtration for %cohomology 
$H(-)$ that is strictly respected by morphisms (coming from $\cu$). In particular, if   
(all) $T(H,-)$ degenerate for $H$ being the (zeroth) homology functor $\cu\to\hrt$, then we obtain a certain weight filtration for $\hrt$.
The degeneration of $T$ in this case is strictly weaker than the transversality (of $w$ and $t$) and the author does not have a complete understanding of this condition. %yet it is sufficient for $w$ to induce a {\it weight filtration} on $\hrt$. The  advantage of this condition 
Its advantage is that it could be 'checked at $t$-exact conservative realizations' of $\cu$. Conjecturally, this statement can be applied to (relative) Voevodsky's  motives and their \'etale realization.

In a subsequent paper \cite{bmm} the notion of transversal structures allows us to apply a certain 'gluing' argument, that reduces the existence of a 'nice' motivic $t$-structure for Beilinson motives over   (equi-characteristic) 'reasonable' base schemes   (cf. \cite{brelmot} and \cite{degcis}) to the case of motives over algebraically closed fields.
The argument mentioned also relies on the degeneration of (Chow)-weight spectral sequences for 'perverse \'etale homology' (which conjecturally implies the corresponding degeneration for $H^0_{t_{MM}}$).

The results of the current paper are somewhat formal (since it does not seem that we obtain much new information on the examples described in the paper); they also do not seem to be really unexpected. Yet this paper is definitely the first one where the weights for the heart of a $t$-structure were related with weight structures and weight spectral sequences; this makes it quite important (at least) for the study of various triangulated categories of motives (since one has certain Chow weight structures for them) and of their realizations. Besides, it  seems that our main setting (of transversal weight and $t$-structures) has not been axiomatized previously.

{\bf Caution on signs of weights.} When the author defined weight structures (in \cite{bws}), he chose $\cu^{w\le 0}$ to be stable with respect to $[1]$ (similarly to the usual convention for $t$-structures); in particular, this meant that for $\cu=K(B)$ (the homotopy category of cohomological complexes) and for the 'stupid' weight structure for it (see \S1.1 of ibid.),   a complex $C$ whose only non-zero term is the fifth one was 'of weight $5$'. Whereas this convention seems to be quite natural, for weights of mixed Hodge complexes, mixed Hodge modules (see Proposition 2.6 below), and mixed complexes of sheaves (see Proposition 3.6.1 of \cite{brelmot}) 'classically' exactly the opposite convention was used (by Beilinson, Saito and others; so, if we extend this convention to complexes, then our $C$ would have weight $-5$). For this reason, in the current paper we use the 'reverse' convention for signs of weights, that is compatible with the  'classical weights' (this convention for the Chow weight structure for motives was already used in \cite{hebpo},  in \cite{wildat}, and in \cite{bzpc}); so the signs of weights used below will be opposite to those in \cite{bws}, \cite{bger}, and in \cite{bzp}, as well as to those in the current versions of \cite{bsurv},   \cite{brelmot}, and \cite{bmm}. 
In particular, for the current convention we have:  $\cu^{w\le i}=\cu^{w\le 0}[i]$, $\cu^{w\ge i}=\cu^{w\ge 0}[i]$ (and so $\cu^{w=i}=\cu^{w=0}[i]$), whereas $\cu^{w\le 0}\perp \cu^{w\ge 1}$.

Now we list the contents of the paper.

In the first section we prove the equivalence of nine %??!
 definitions of transversal weight and $t$-structures. This yields several relations between $t$-structures, weight structures, weight filtrations, and {\it semi-orthogonal} generators $\aui$ in this situation.  We don't recall the (general) theory of weight structures %(and its definitions)
 in this paper; so a reader that is not acquainted with it should probably  consult \cite{bws}  or  \cite{bsurv}  (yet paying attention to the Caution above!). In the end of the section we also calculate the $K_0$-group of a triangulated category endowed with transversal weight and $t$-structures.

We start the  second section by noting that the results of \cite{bvk} yield a general criterion for existence of a weight structure that is transversal to the canonical
$t$-structure for $D^b(\au)$ if $\au$ admits a  'weight filtration' %cf. ??!!!
(with semi-simple 'factors'). We use this result for the construction of the main examples (of transversal weight and $t$-structures) in this paper; yet cf. Remark \ref{rfun}(1).  
Applying some more results of \cite{bvk}, we deduce the existence of a weight structure that is transversal to the canonical 
 (i.e. 'motivic') $t$-structure for the derived category $\dom$ of Deligne's $1$-motives. Then we prove that  the functors $\lalb,\rpic\to\dmge
 \subset \dmgm$ (considered in \cite{bvk}) 
 respect the weight structure constructed (this should also be true %for adjoint functors also??!!!
for the corresponding motivic $t$-structures if one %??!!
 could define $t_{MM}$ for $\dmge$). Note that this fact is not automatic, since
 the usual definition of weights for $1$-motives does not mention Chow motives. Next we verify that on the derived category of mixed Hodge modules over a complex variety $X$, and on the 
Beilinson's derived category $\dbhp$ of graded polarizable mixed Hodge complexes (over a base field $k\subset \mathbb{C}$) there exist  weight structures
transversal to the corresponding canonical $t$-structures; the singular realization functor $\dmgm(k)\to \dbhp(k)$ respects the corresponding weight structures.  Lastly, we describe the conjectural relations between various 'structures' for $\dmgm$.  %These weight structures correspond to 'classical' weights for mixed Hodge complexes??!!

In \S3 we recall weight filtrations and weight spectral sequences for (co)homology (that correspond to weight structures). We 
relate the degeneration of weight spectral sequences  (at $E_2$) with the 'exactness' of the corresponding weight filtration. In the case when the homology is the %one coming from $t$-truncations
$t$-one, we obtain a certain weight filtration for $\hrt$ (though this degeneration condition is strictly weaker than the transversality of $t$ and $w$).

 %This observation will also be applied for the study of relative motives in the subsequent paper (that was mentioned above).  

The author is deeply grateful to prof. L. Barbieri-Viale who
inspired him to write this paper, and to prof. L. Positselski for interesting comments.

\begin{notat}

$\cu$ below will always denote some triangulated category.
$t$ will %always
 denote a bounded $t$-structure for $\cu$,
and $w$ will be a bounded weight structure for it (the theory
of weight structures was studied in detail in \cite{bws}; note the Caution above, and see also \cite{konk} and \cite{bsurv}). %; we will .

For $X\in \obj \cu$, $i\in \z$, we will consider the following distinguished triangles:
\begin{equation}\label{etd}
\tau_{\le i}X\to X\to \tau_{\ge i+1}X\end{equation}
 and
\begin{equation}\label{ewd} w_{\le i}X\to X\to w_{\ge i+1}X \end{equation} that come from the $t$-decompositions of $X[i]$ shifted by $[-i]$ (resp. from a weight decomposition of $X[-i]$ shifted by $[i]$), i.e. $ \tau_{\le i}X\in \cu^{t\le i}$, $ \tau_{\ge i+1}X\in \cu^{t\ge i+1}$,
 $ w_{\le i}X\in \cu^{w\le i}$, $ w_{\ge i+1}X\in \cu^{w\ge i+1}$; cf. Remark 1.2.2 of \cite{bws}).

 $X^{\tau=i}\in\cu^{t=0}$ will denote the $i$-th cohomology of $X$ with respect to $t$
  i.e. the cone of the corresponding morphism $\tau_{\le -1}(X[i])\to \tau_{\le 0}(X[i])$; $\tau_{=i}X=X^{\tau=i}[-i]$;  $\hrt$ will denote the heart of $t$. %; $\cu^{t=0}=\obj \hrt$.

$D\subset \obj \cu$ will be
called {\it extension-stable}
    if for any distinguished triangle $A\to B\to C$
in $\cu$ we have: $A,C\in D\implies
B\in D$. Note that $\cu^{t\le i}$, $\cu^{t\ge i}$, $\cu^{t=i}=\cu^{t\ge i}\cap \cu^{t\le i}$, %$\cu^{w=0}$,
 $ \cu^{w\ge i}$, $\cu^{w\le i}$, $\cu^{[i,j]}=\cu^{w\ge i}\cap \cu^{w\le j}$, and $\cu^{w=i}=\cu^{[i,i]}$ 
 are extension-stable for any $t,w$ and any $i\le j
 \in \z$.

For a subcategory $H\subset \cu$ we will call   the smallest %strict
extension-stable subcategory of $\cu$ containing $H$ the {\it envelope} of $H$ (in $\cu$). %$H\cup \ns$.

For $D,E\subset \obj \cu$ we will write $D\perp E$ if $\cu(X, Y)=\ns$
 for all $X\in D,\ Y\in E$.

For $B\subset \cu$ we will call  the subcategory of $\cu$ whose objects are  all
retracts of  objects of $B$ (in $\cu$) the {\it Karoubi-closure} of $B$
in $\cu$.

  For a class of
objects $C_i\in\obj\cu$, $i\in I$, we will denote by $\lan C_i\ra$
the smallest strictly full triangulated subcategory containing all $C_i$; for
$D\subset \cu$ we will write $\lan D\ra$ instead of $\lan C:\ C\in\obj
D\ra$.

$\au$ will always be an abelian category;  $\au_i$ (for $i$ running through all integral numbers) will always be additive, and will often be abelian semi-simple.

$k$ will be our perfect base field (sometimes it will be contained in or equal to the field of complex numbers).

\end{notat}

 \section{Transversal weight and \texorpdfstring{$t$}{t}-structures: the general case}

In \S\ref{saux} we prove  several auxiliary statements; none of them are really new. We introduce our main formalism (of transversal weight and $t$-structures) and study it in \S\ref{sdtrans}.

\subsection{Auxiliary statements}\label{saux} %; weight filtrations for triangulated categories}

We will need the following easy homological algebra statements.

\begin{lem}\label{lbas}

1. Let  $T:X\stackrel{a}{\to} A\stackrel{f}{\to} B\stackrel{b}{\to} X[1]$
and $T':X'\stackrel{a'}{\to} A'\stackrel{f'}{\to}
 B'\stackrel{b'}{\to} X'[1]$ be
distinguished triangles.

  Let $B\perp A'[1]$.
 Then for any morphism $g:X\to X'$ there exist $h:A\to A'$ and
$i:B\to B'$ completing $g$ to a morphism of triangles $T\to T'$.

Moreover, if $B\perp A'$, then  $g$ and $h$ are unique.

2. In particular, for any $i\in\z$, $X,Y\in\obj \cu$, any  $g\in \cu(X,X')$ could be completed to a morphism of distinguished triangles
\begin{equation}\label{efwd}
 \begin{CD} w_{\le i} X@>{}>>
X@>{}>> w_{\ge i+1}X\\
@VV{}V@VV{g}V@ VV{}V \\
w_{\le i} Y@>{}>>
Y@>{}>> w_{\ge i+1}Y \end{CD}
\end{equation}

3. If $D\perp E$ ($D,E\subset \obj\cu$), then the same is true for their envelopes.

4. Let $D,E\subset \obj\cu$ be extension-stable, $D\perp (E\cup E[1])$.
For some $F\subset \obj \cu$ suppose that for any $X\in F$ there exists a distinguished triangle $Y\to X\to Z$ with $Y\in D,\ Z\in E$. Then such a distinguished triangle also exists for any $X$ belonging to the envelope of $F$ (in $\cu$).

5. For any $i\le j\in \z$  we have: $\cu^{[i,j]}$ is the envelope of
 $\cup_{i\le l\le j}\cu^{w=l}$ in $\cu$.

6. Let $B$ be an additive category; let $B_i\subset B$, $i\in \z$, be  its full additive subcategories such that $B_i\perp B_j$ for $j>i$, and $\obj B=\bigoplus_{i\in \z} \obj B_i$. Suppose that all $B_i$ are idempotent complete (i.e. that for any $X\in\obj B_i$
and any idempotent
$s\in B_i(X,X)$ there exists a decomposition $X=Y\bigoplus Z$ such that
$s=j\circ p$, where $j$ is the inclusion $Y\to X(\cong Y\bigoplus Z)$, $p$ is the projection $X(\cong Y\bigoplus Z)\to Y$).
Then $B$
is idempotent complete also. 

\end{lem}
\begin{proof}

1. This is Lemma 1.4.1 of \cite{bws}; it follows immediately from Proposition 1.1.9 of
\cite{bbd}.

2. Follows immediately from assertion 1; cf. Lemma 1.5.1 of \cite{bws}.

3. Very easy; note that for any $X\in \obj \cu$ the (corepresentable) functor $\cu(X,-)$
is homological, whereas $\cu(-,X)$ is cohomological.

4. See Remark 1.5.5 of \cite{bws} or Proposition 1.8 of \cite{hebpo}.

5. Easy from Proposition 1.5.6(2) of \cite{bws}. % ibid. %(note that one can take).

6. We prove the assertion in question for $X=\bigoplus_{i_1\le i \le i_2} X_i$ where $X_i\in \obj B_i$, by induction in $i_2-i_1$. For $i_2-i_1\le 1$ the statement follows from our assumptions.

%conv??!!

Suppose now that our claim holds if $i_2-i_1< m$ for some $m>1$. Let $i_2-i_1=m$. 
We decompose $s$ as $\bigoplus s_{ij}$, $s_{ij}\in B(s_i,s_j)$. Our orthogonality assumption yields that the morphism $s_1=s_{i_2,i_2}:X_{i_2}\to X_{i_2}$ is idempotent, as well as $s_2=\bigoplus_{i_1\le i,j< i_2}s_{ij}$.
The inductive assumption yields that $s_1$ and $s_2$ correspond to certain $X_1,Y_1,X_2,Y_2\in \obj B$ respectively. 

Now, it  suffices to verify that the morphism $s$ is conjugate to $s_1+s_2$. Denote $s-s_1-s_2$ by $d$. Then our orthogonality assumptions yield: $s_1\circ s_2=s_2\circ s_1=d\circ  s_2=s_1\circ d=d^2=0$. Besides, $d=d\circ s_1+s_2 \circ d$; composing this with $s_2$ %from the left
 we obtain that  $s_2\circ d\circ s_1=0$. Hence $(\id_X-d\circ s_1+s_2\circ d)(\id_X+d\circ s_1-s_2\circ d)=\id_X$; therefore for $h=\id_X-d\circ s_1+s_2\circ d$ we have $h\circ  (s_1+s_2)\circ h^{\ob}=s$. %???!!

%??!! 

\end{proof}

Below we will need a certain class of 'nice' weight decompositions.

\begin{defi}\label{dnwd}
For some $\cu,t,w$ we will say that a distinguished triangle (\ref{ewd}) (for some $i,X$) is {\it nice} if 
$w_{\le i}X,X,w_{\ge i+1}X\in \cu^{t=0}$. 

We will also say that this distinguished triangle is a {\it nice decomposition} of $X$ (for the corresponding $i$), and that the morphism $w_{\le i}X\to X$ extends to a nice decomposition.
\end{defi}

Now we %prove
formulate a simple implication of  Lemma \ref{lbas} (we will use a very easy case of it below, and a somewhat more complicated one in \cite{bmm}). %(4). Remark??!! 

\begin{lem}\label{lext}
We fix some $\cu,w,t,i$; suppose that for a certain $N\subset \cu^{t=0}$ a nice %choice of (\ref{ewd}) 
decomposition exists for any $X\in N$. Consider $N'\subset \cu^{t=0}$ being the smallest subclass containing $N$ that 
satisfies the following condition: if $A,C\in N'$, $$A\stackrel{f}{\to}B\stackrel{g}{\to}C$$
is a complex (i.e. $g\circ f=0$), $f$ is monomorphic, $g$ is epimorphic, $ \ke g/\imm f\in N'$, then $B\in N'$. Then a nice %choice of (\ref{ewd}) 
decomposition exists for any $X\in N'$ (and the same $i$).

%??!! + the exactness argument below that uses Lemma 1.5.4 of \cite{bws}??!!

%is extension-stable a
\end{lem}

\begin{proof}
It suffices to note that $N'$ %belongs to 
is exactly the smallest extension-stable subcategory of $\cu$ containing $N$, and apply Lemma \ref{lbas}(4).
\end{proof}

Next we study certain ('weight') filtrations of triangulated categories.

 \begin{defi}\label{dsosc}
1. We will say that a family $\{\au_i\},\ \au_i\subset \cu,\ i\in\z$ %of additive %extension-stable?? automatic in our case??! subcategories
is {\it semi-orthogonal} if $\au_i\perp \au_j[s]$ for any $i,j,s\in \z$ such that $s<0$, or $s>i-j$.

We will say that $\{\au_i\}$ are {\it strongly semi-orthogonal}
 if we also have  $\au_i\perp \au_j$ for any   $i> j$ (and so, for any $i\neq j$). %cohomology of Chow motives is in negative t-degrees??!!

2. We will say that $\{\au_i\}$ is generating (in $\cu$) if 
$\lan \cup_i \obj A_i\ra=\cu$.
%there does not exist a strict triangulated $\cu'\subset \cu$, $\cu'\neq \cu$, such that all $\au_i\subset \cu'$. %idempotents??! Possible, but not necessary??!!!
\end{defi}

%In the proof of the Theorem
%We will need the following (easy, and not at all original) lemma.

Now we prove that a semi-orthogonal generating family yields a certain ('weight') filtration  for $\cu$ in the sense of Definition E17.1 of \cite{bvk}, and study its properties.

\begin{lem}\label{lwf}
Let $i> j\ge n\in \z$.

I  Suppose that $\{\au_s\subset \cu\}$ is a semi-orthogonal family; denote $\lan \au_s\ra$ by $\cu_s$ (for any $s\in \z$).

Then $\cu_j\perp \cu_i$.

II Let $\cu_l\subset \cu$ for $l\in
\z$ be triangulated; suppose that $\cu_l\perp \cu_m$ if $l<m$.

For any $r\le q\in \z$ denote $\lan \cup_{r\le s\le q}\obj \cu_{s} \ra$
by $\cu_{[r,q]}$, and   denote $\lan \cup_{s\le r}\obj \cu_{s} \ra$ by $\cu_{\le r}$.

Then the following statements are fulfilled.

1.  For any $X\in \obj\cu_{[n,i]}$ there exists a distinguished triangle
\begin{equation}\label{ewf}
X_1\to X\to X_2
\end{equation}
  such that $X_1\in \obj \cu_{[n,j]}$, $X_2\in \obj \cu_{[j+1,i]}$. More
  generally, for $X\in \cu_{\le i}$ one can find (\ref{ewf}) with $X_1\in \obj \cu_{\le j}$.

  Besides, this triangle is (canonical and) functorial in $X$ (in both cases).

2. The embedding $\cu_{[j+1,i]}\to \cu_{\le i}$ possesses an exact %right
left adjoint $a_{i,j}$; the 'kernel' of $a_{i,j}$ is exactly $\cu_{\le j}$.

3. Suppose that $\{\cu_l\}$ are generating. Then  the embedding
 $\cu_{\le i}\to \cu$ possesses an exact right adjoint $b_i$.

%3. Dually???!! Adjoint to $\cu_{\ge j}\to $\cu_{\ge i}$??!!

%Then the inclusion
\end{lem}
\begin{proof}
I If $i<l$, then $\au_i[r]\perp \au_l$ for any $r\in \z$ (by Definition \ref{dsosc}). Now the result is immediate from Lemma \ref{lbas}(3).

II 1. Since $w$ is bounded, we have $\cu_{\le i}=\cup_{m\le i}\cu_{[m,i]}$; hence it suffices to verify the existence of (\ref{ewf}) for $X\in \cu_{[n,i]}$. 

%We verify that assertion 1 implies assertion 2 (for any $\{\cu_i\}$ satisfying the corresponding orthogonality condition).

We have a 'trivial' example of (\ref{ewf}) if $X\in \cu_l$ for any $i\ge l\ge n$.
Hence the existence of (\ref{ewf}) in general is immediate from Lemma \ref{lbas}(4). Now, any morphism $X\to X'$ could be uniquely extended to a morphism of the corresponding triangles by Lemma \ref{lbas}(1).
Hence we obtain the functoriality of (\ref{ewf}).

2,3: Immediate from assertion II1 by well-known homological algebra statements; see %see also??! %conv
Proposition E.15.1 of \cite{bvk}.

\end{proof}

\subsection{Transversal weight and \texorpdfstring{$t$}{t}-structures: equivalent definitions and their consequences}\label{sdtrans}

\begin{theo}\label{transmain}

The following conditions are equivalent.

(i) There exists a strongly semi-orthogonal  generating family $\{\aui\}$ in $\cu$ such that all of $\aui$ are abelian semi-simple.

(ii) There exists a semi-orthogonal  %generating 
family $\{\aui\}$ in $\cu$ such that for  $\cu^{t\le 0}$ (resp. $\cu^{t\ge 0}$) being the envelope of $\cup_{i\in\z,j\ge 0}\aui[j]$  (resp. of $\cup_{i\in\z,j\le 0}\aui[j]$) we have: $(\cu^{t\le 0}, \cu^{t\ge 0})$ yield a $t$-structure for $\cu$.

(ii') There exists a semi-orthogonal  
 family $\{\aui\}$ in $\cu$ such that the envelope of $\cup_{i\in\z}\aui$ yields the heart of a certain $t$.

 %$t$ is bounded??!!

(iii) There exist a $t$ and a semi-orthogonal family $\{\aui\subset \hrt\}$ that satisfy the following condition: for each $X\in \cu^{t= 0}$ 
there exists an exhaustive separated
increasing filtration by subobjects $W_{\le i}X$, $i\in \z$, such that $W_{\le i}X/W_{\le i-1}X$ belongs to $\obj\aui$ for all $i\in \z$.

(iii') %For $t$ 
The filtration (of any $X\in \cu^{t=0}$) described above exists and is $\hrt$-functorially determined by $X$.

(iv) There are $t,w$ for $\cu$ such that for any $X\in \cu^{t=0}$, $i\in \z$, there exists a nice %choice of (\ref{ewd}) 
decomposition (see Definition \ref{dnwd}). % with $w_{\ge i+1} X$ and $w_{\le i}X$ belonging to $\cu^{t=0}$.
%nice choice of (\ref{ewd})??!!

(iv') %The nice choice of (\ref{ewd}) 
Nice decompositions  exist, and they are also $\hrt$-functorial in $X$ (if we fix $i$); the corresponding  functors $X\mapsto w_{\le i} X$ 
and $X\mapsto w_{\ge i+1} X$ are exact (on $\hrt$).

(v) There are $t,w$ for $\cu$ such that for any $X\in \cu^{t=0}$, $i\in \z$, there exists a choice of $w_{\le i}X$ 
such that the morphism $\imm ((w_{\le i}X)^{\tau=0}\to X)\to X$ extends to a nice decomposition of $X$.

(v') For  $t,w$ and any $X,i,w_{\le i}X$ (as above) the morphism $\imm ((w_{\le i}X)^{\tau=0}\to
X)\to X$ extends to a nice decomposition of $X$.

%(v) There exist $t,w$ such that $\aui=\cu^{t=0}\cap \cu^{w=i}$ generate $\cu$.

\end{theo}

 %Lemma: common weight filtration??!!
\begin{proof}%[Proof of Theorem \ref{transmain}]
Certainly, (ii') implies (ii) (since $t$ is bounded), (iii') implies (iii),  (iv') implies (iv), and  (v') implies (v).

%Now we verify that
{\bf (i) $\implies$ (ii).}

Semi-orthogonality yields that the 'generators' of  $\cu^{t\le 0}[1]$ are orthogonal  to those of $\cu^{t\ge 0}$; hence Lemma \ref{lbas}(3) yields:  
 $\cu^{t\le 0}[1]\perp \cu^{t\ge 0}$.

It remains to verify the existence of $t$-decompositions.

For any $i\in \z$ we have $\cu_i\cong\bigoplus_{j\in \z} \au_i[j]$ (in the notation of loc.cit.). Indeed, the obvious comparison functor $\bigoplus_{j\in \z} \au_i[j]\to \cu_i$ is an equivalence of triangulated categories, since $\au_i$ is semi-simple and $\aui\perp\aui[j]$ for any $j\neq 0$.

Hence any object of $\cu_i$ admits a $t$-decomposition $X\cong \tau_{\le 0}X\bigoplus \tau_{\ge 1}X$ whose components also belong to  $\cu_i$. 

 Now, it suffices to verify: if for some $j<i$ any object of $\cu_{[j+1,i]}$ admits a $t$-decomposition inside $\cu_{[j+1,i]}$, 
 then a similar statement holds for any $X\in \obj\cu_{[j,i]}$.

 Lemma \ref{lwf}(II1)
  yields the existence of a distinguished triangle $X_1\to X\to X_2\stackrel{g}\to X_1[1] $
  such that $X_1\in \obj \cu_{j}$, $X_2\in \obj \cu_{[j+1,i]}$. %We also obtain that $X_1\in \obj cu_{[i+1,j]}$. %Why??!

  Now we argue  as in the proof of Lemma 1.5.4 of \cite{bws}.
We can complete $g$ to a morphism of distinguished triangles
\begin{equation}\label{d2x3}
\begin{CD}
\tau_{\le 0}X_2@>{%a_2
}>>X_2@>{%b_2
}>>\tau_{\ge 1}X_2 \\
@VV{}V@VV{g}V @VV{}V\\
(\tau_{\le 0}X_1)[1]@>{%a_1[1]
}>>X_1[1]@>{%b_1[1]
}>>(\tau_{\ge 1}X_1)[1]\\
\end{CD} \end{equation}
Indeed, by Lemma \ref{lbas}(1) it suffices to verify that $\tau_{\le 0}X_2\perp (\tau_{\ge 1}X_1)[1]$; the latter easily follows from the strong semi-orthogonality of $\{\au_s\}$ (see Lemma \ref{lbas}(3)). %more??!!

 %the composition morphism $\tau_{\le 0}X\to (\tau_{\ge 2}X_1)[1]$ necessarily vanishes
%$\tau_{\le 0}X_2\perp (\tau_{\ge 2}X_1)[1]$ by the orthogonality for $t$ (that we have already proved); see Lemma 1.4.1 of \cite{bws}.
%$t$-decompositions are functorial.

Moreover, we can complete $g$ to the following diagram (starting from the left hand side square of (\ref{d2x3}), and using Proposition 1.1.11 of \cite{bbd}):
 \begin{equation}\label{dia3na3}
\begin{CD}
\tau_{\le 0}X_2@>{%a_2
}>>X_2@>{%b_2
}>>\tau_{\ge 1}X_2 \\
@VV{}V@VV{g}V @VV{}V\\
(\tau_{\le 0}X_1)[1]@>{%a_1[1]
}>>X_1[1]@>{%b_1[1]
}>>(\tau_{\ge 1}X_1)[1]\\
 @VV{}V@VV{}V @VV{}V\\
Y[1]@>{}>>X[1]@>{}>>Z[1]\\
 @VV{}V@VV{}V @VV{}V\\
(\tau_{\le 0}X_2)[1]@>{%a_2
}>>X_2[1]@>{%b_2
}>>(\tau_{\ge 1}X_2)[1] \\
\end{CD}
\end{equation}
such that all rows and columns are distinguished triangles, and all
squares are commutative. Therefore the extension-stability of $\cu^{t\le 0}$ yields that it contains $Y$;  
the extension-stability of $\cu^{t\ge 1}$ yields that it contains $Z$;  hence $Y\to X\to Z$ is the $t$-decomposition desired.

{\bf (ii) $\implies$ (iv).}
%Now we verify that (ii) implies (iv). %(iv'). %(iii')??!!
 %The semi-orthogonality condition implies that we can apply Remark 1.5.5 of \cite{bws} again here. So
 We take $C_1$ being the envelope of $\{\aui[j],\ i+j \ge 0,\ i,j\in \z\}$ in $\cu$, %(see the Notation),
 $C_2$ being the envelope of 
 $\{\aui[j],\ \ i+j\le 0\}$. Note that
 %$\cu^{w\le 0}$ 
 $C_1$ is the envelope of $\{H[j],\ j\ge 0\}$, %$\cu^{w\ge 0}$
 $C_2$ is the envelope of $\{H[j],\ j\le 0\}$, where $\obj H=\bigoplus_{i\in \z}\obj\aui[i]$.
 Besides, $H$ is negative i.e. $H\perp H[j]$ for all $j>0$.

 Hence, as shown (in the proof of) Theorem 4.3.2(II) of \cite{bws}, there exists a weight structure $w$ (for $\cu$) such that $\cu^{w\ge 0}$ (resp. $\cu^{w\le 0}$) is the Karoubi-closure of $C_1$ (resp. $C_2$) in $\cu$. 
 yield a bounded weight structure $w$ for $\cu$ (actually, this is a simple consequence of Lemma \ref{lbas}(3,4)).
  Moreover, the heart of this weight structure is the idempotent completion of $H$. Now, $H$ is idempotent complete itself,
   since all $\au_i$ are (note that $\aui[-i]\perp \au_j[-j]$ for $j>i$, hence we can apply Lemma \ref{lbas}(6)). %to Lemma \ref{lbas}???!!!
 Therefore $H=\hw$.
 %Since any bounded object of $\cu$ possesses a bounded Postnikov tower (whose 'factors' are objects of $\hw=H$; see ), we obtain
 Then Lemma \ref{lbas}(5) implies  that $\cu^{w\ge 0}=C_1$ and $\cu^{w\le 0}=C_2$ (i.e. we don't need   Karoubi-closures here; 
 here we use the fact that  $\cu^{w\le 0}=\cup_{i\le 0}\cu^{[i,0]}$ and  $\cu^{w\ge 0}=\cup_{i\ge 0}\cu^{[0,i]}$ for a bounded $w$).
 Lastly, by %Lemma 1.5.4 of \cite{bws},
Lemma %\ref{lbas}(4)
\ref{lext}
 it suffices to verify the existence of %filtrations
 nice %weight
  decompositions for those objects of $\hrt$ that belong to one of $\au_i$; this is obvious.

{\bf (v) $\implies$ (iv)}: obvious.

{\bf (iv) $\implies$ (v').} It suffices to note: by
Proposition 2.1.2(1) of \cite{bws}, $\imm ((w_{\le i}X)^{\tau=0}\to X)$ does not
depend on the choice of $w_{\le i}X$ (cf. Definition \ref{dwfil} below). Hence it suffices to consider the case
when $w_{\le i}X$ comes from a nice decomposition of $X$, and then the
statement is obvious.

 %Now %, suppose that
 %we have (iv). %nado?? From (i) directly??!!
%we verify that 
{\bf (iv) $\implies$  (iv') and (iii).} %(iii').

We set $\aui=\cu^{t=0}\cap \cu^{w=i}$. 
The orthogonality properties of  weight and $t$-structures immediately yield that $\au_i$ are semi-orthogonal.
%By the orthogonality properties of  weight and $t$-structures we obtain that  $\au_i\perp \au_j[s]$ for any $i,j,s\in \z$ such that $s<0$ or $s>j-i$.

Now we prove (iii). Since %for a nice %choice of 
%(\ref{ewd}) 
all terms of a nice %(\ref{ewd}) 
decomposition belong to $\cu^{t=0}$, it yields a short exact sequence in $\hrt$. In particular, the corresponding morphism  $w_{\le i}X\to X$ is monomorphic in $\hrt$.

Now suppose that $X\in \cu^{w\le i}$. Then we have $w_{\ge i}X\in \cu^{w=i}$; see Proposition 1.3.3(6) of \cite{bws}. Hence, %any 'nice' choice of 
$w_{\ge i}X$ belongs to $\aui$ for any nice (\ref{ewd}).
%conv

Loc.cit. also yields: if $X\in \cu^{w\ge j}$, $j<i$, then any choice of $w_{\le i}X$ belongs to $\cu^{w\ge j}$ also. Hence for $X\in \cu^{[r,s]}\cap \cu^{t=0}$, $r\le s\in \z$, one can take
$W_{\le l}X=X$ for $l\ge s$, then by induction starting from $i=s-1$ down to $i=r$  take a %'nice' 
choice of $W_{\le i}X$ coming from a nice %choice of (\ref{ewd}) for 
decomposition of $W_{\le i+1}X$ (that was constructed on the previous step), and set $W_{\le l}=0$ for $l<r$; this filtration would satisfy the conditions of (iii).

Now we verify (iv'). Any morphism in $\cu$ could be extended to a morphism of (any choices of) weight decompositions by  Lemma \ref{lbas}(2). Moreover, this extension is unique in our case by parts 1 and 3 of loc.cit. (here we apply the orthogonality statement proved above). Hence we obtain that nice %weight decompositions (coming) 
choices of (\ref{efwd}) (for a fixed $i$) yield a functor (here we take $X\in \cu^{t=0}$).

Now, Lemma 1.5.4 of \cite{bws} yields %@@@prove; modification???!!!
 that for any distinguished triangle $A\to B\to C$ in $\cu$, any triangles (\ref{ewd}) for $A,C$ could be completed to a diagram
\begin{equation}\label{dia3na3n}
\begin{CD}
w_{\le i}A@>{}>>A@>{}>>w_{\ge i+1}A \\
@VV{}V@VV{g}V @VV{}V\\
w_{\le i}B@>{}>>B@>{}>>w_{\ge i+1}B \\
 @VV{}V@VV{}V @VV{}V\\
w_{\le i}C@>{}>>C@>{}>>w_{\ge i+1}C
\\
\end{CD}
\end{equation} %to Lemma \ref{lbas}??!!
all of whose rows and columns are distinguished triangles (and the middle row is given by some choice of  (\ref{ewd}) for $B$).
Applying this fact for $A,B,C\in \cu^{t=0}$ and nice %choices of (\ref{ewd}) %weight 
decompositions of  
%for
$A,C$, we obtain that the middle row %yields a nice  (\ref{ewd}) % weight decomposition of 
is a nice decomposition of $B$ (since $\hrt$ is extension-stable in $\cu$). Then the exactness of columns (in $\hrt$)  concludes the proof of (iv').

Next we note that (ii) along with (iii) implies (ii'). Indeed, the envelope of $\aui$ obviously lies in $\hrt$, whereas (iii) %(along with the description of $\aui$ given above)
 yields that this inclusion is an equality.

%Now, we verify (iv').
  %more??!!
 %Hence we (iv) implies (iv').

 It remains to verify that (iii) implies (iii') and (i).

 To this end first we verify that (iii) implies (iv).
 Obviously, the family $\{\aui\}$ is generating (since $t$ is bounded).
 We consider $C_1,C_2\subset \cu$ introduced in the proof of (ii) $\implies$ (iv). As we have already noted above, the Karoubi-closures of $C_1$ and $C_2$ in $\cu$ yield a weight structure for $\cu$. Hence the distinguished triangles coming from 
 the short exact sequences $0\to W_{\le i}X\to X\to X/W_{\le i}X\to 0$ yield (\ref{ewd}). We obtain that (iii) implies (iv); hence (iii) also yields (iv'). 
 
 Obviously, (iv') implies (iii'). Also, (iv') yields that $\au_i\perp \au_j$ for $j>i$ by the  Remark E7.8 and Proposition E7.4(4) of \cite{bvk} (cf. Proposition \ref{pdab} below).

So, it remains to prove that $\aui$ are abelian semi-simple. We verify that (for a fixed $i\in \z$) the classes $\obj \cu_i\cap \cu^{t\le 0}$ and $\obj \cu_i\cap \cu^{t\ge 0}$ yield a $t$-structure for $\cu_i$ (i.e. that $t$ could be restricted to $\cu_i$).
To this end we note that 
  for any $i\ge j\in \z$, $X\in \cu_{\le i}$ (see Lemma \ref{lwf}(II)) 
  %any choice of a 
  the distinguished triangle $W_{\le j}X\to X\to X/W_{\le j}X$ (as considered above)
  is simultaneously a choice of (\ref{ewf}). 
    %: essentially we have just proved 
  It easily follows that all $b_j$ and $a_{i,j}$ %(see Lemma \ref{lwf}(II)) 
  respect $\hrt$. Hence they also respect $t$-decompositions; see Lemma E19.1 of \cite{bvk}. Hence  applying $a_{[i+1,i]}\circ b_i$ to the $t$-decomposition of $X\in \cu_i$ (see (\ref{etd})) we obtain that its components belong to $\cu_i$. We also obtain that the heart of this $t$-structure is $a_{[i,i-1]}\circ b_i(\hrt)=\aui$. Hence $\aui$ is an abelian category, and short exact sequences in it yield distinguished triangles in $\cu$. Therefore $\aui$ is abelian semi-simple, since $\aui\perp\aui[1]$ by semi-orthogonality.

\end{proof}

\begin{defi}\label{dtrans}
If $w,t$ satisfy %condition (iv') of the theorem (or any other condition),
the (equivalent) conditions of the theorem, we will say that
$t$ is transversal to $w$.
\end{defi}

\begin{rema}\label{rtf} %$H_i'\subset \hrt$??!!! Non-orthogonality interval 'moves to the positive numbers'?!!! Adjust notation to fit with \cite{bvk}??!!!

1. One could (try to) modify the conditions of the Theorem so in order to include the case when $w$ and $t$ are not (necessarily) bounded. Yet to this
end one would definitely require certain technical restrictions on $\cu$ (cf. Theorems 4.3.2 and
4.5.3 of \cite{bws}).

2. Some more details on connections between $w$, $t$, and $\{\aui\}$ are contained in
the proof of the Theorem. In particular, note that %$t$
the functor $X\mapsto \bigoplus_{i\in \z}\tau_{=i}X$ yields a splitting $\cu^{w=0}=\bigoplus \obj \aui[-i]$ (though we don't have an isomorphism of the corresponding categories, since there could be non-zero morphisms from $\aui[-i]$ into $\au_j[-j]$ for $j<i$).  Besides, $\hrt$ %
possesses a separated exhaustive filtration with semi-simple factors $\aui=\hrt\cap \hw[i]$.
%is necessarily noetherian and artinian.

3. Condition (i) of the Theorem is self-dual. If follows: if $w,t$ are transversal for $\cu$, then the structures $w^{op},t^{op}$ for the opposite category $\cu^{op}$ are transversal also. The latter structures are defined as follows: $\cu^{op,w^{op}\le 0}=\cu^{w\ge 0}$ and  $\cu^{op,w^{op}\ge 0}=\cu^{w\le 0}$; $\cu^{op,t^{op}\le 0}=\cu^{t\ge 0}$ and  $\cu^{op,t^{op}\ge 0}=\cu^{t\le 0}$ (cf. Remark 1.1.2(1) of \cite{bws}). 

Besides, for any $i,j\in \z$ the structures $w[i], t[j]$ are also transversal; here
$\cu^{w[i]\le 0}=\cu^{w\le i}$ and $\cu^{w[i]\ge 0}=\cu^{w\ge i}$; 
$\cu^{t[j]\le 0}=\cu^{t\le j}$ and $\cu^{t[j]\ge 0}=\cu^{t\ge j}$. 

4. Proposition 2.1.2(1) of \cite{bws} actually yields (cf. %the proof of the Theorem and 
Definition \ref{dwfil} below) that for {\bf any} $t,w$ the correspondence $X\mapsto \imm ((w_{\le i}X)^{\tau=0}\to X)$ yields a functor in $\hrt\to \hrt$; hence this is also true for $X\mapsto\cok ((w_{\le i}X)^{\tau=0}\to X)$. So, in order to verify that $t,w$ are transversal it suffices to verify that these functors take their values in $\cu^{w\le i}$ and $\cu^{w\ge i+1}$, respectively (for all $i\in \z$).

%5. Alternatively, one could describe $W_{\le i}X$ (given by part (iii') of the Theorem) as $\imm ((w_{\le i}X)^{\tau=0}\to (w_{\le i+1}X)^{\tau=0})$; so $W_{\le i}X$ also comes from the corresponding {\it virtual $t$-truncation} of $M\mapsto M^{\tau=0}$; see \S2.5 of \cite{bws}.

%weight complexes, $E_2T(H,X)$??!!

\end{rema}

Most of the following results were also (essentially) verified in the process of proving Theorem \ref{transmain}.

\begin{pr}\label{pdtr} Let $t$ be transversal to $w$, $i\in \z$. Then the
following statements are fulfilled.

 I1. The functors $X\mapsto \tau_{\le i}X$ and $X\mapsto \tau_{\ge
 i}X$ map $\cu^{w\le 0}$ and $\cu^{w\ge 0}$ into themselves.

 2. $X\in \cu^{w\le i}$ (resp. $X\in \cu^{w\ge i}$) whenever for
 any $j\in \z$ we have $W_{\le i+j}(X^{\tau=j})=X$  (resp. $W_{\le i+j-1}(X^{\tau=j})=X^{\tau=j}$; here $W_{\le i}(-)$ is the  filtration given by condition (iii') of  Theorem \ref{transmain}).

II1. The functor $X\mapsto W_{\le i}X$ (from $\hrt$ into $\hrt$)
given by condition (iii') of the Theorem, is right adjoint to the
embedding $\cu^{w \le i}\cap \hrt \to \hrt$; it is exact.

2. For $X\in \cu^{t=0}$ denote $X/W_{\le i}X$ by  $W_{\ge i+1}X$. Then the functor $W_{\ge i+1}(-)$ is left adjoint to the embedding $\cu^{w \ge +1}\cap \hrt \to \hrt$; it is exact.

III 1. The functors $Gr_i:X\mapsto W_{\ge i} (W_{\le i} X)$  and $Gr_i':X\mapsto W_{\le i} (W_{\ge i} X)$ (defined as the compositions of the functors from assertion II) are canonically isomorphic exact projections of $\hrt$ onto $\aui$. Moreover, $Gr_i(X)\cong W_{\le i} X/ W_{\le i-1} X$.

2. For $X\in \cu^{t=0}$ we have: $X\in \cu^{w\le i}$ (resp. $X\in \cu^{w\ge i}$) whenever $Gr_j(X)=0$ for all $j>i$ (resp. for all $j<i$).

%3 Strict compatibiliy of morphisms with the weight filtration on $\ht$??!! A more general result (for weak weight structures)??!!
\end{pr}
\begin{proof}

I1. Since $t$ is bounded, it suffices to verify a similar statement for the functors $\tau_{=j}$ for all $j\in \z$ (since the functors mentioned in the assertions could be obtained from these functors via 'extensions'). 

Lemma \ref{lbas}(5) allows  reducing the latter statement to its analogue for  $\cu^{w=0}$ (and all $j\in \z$). 
Indeed, note that  for any $l\in \z$ the functor $W_{\le l}$ (see condition (iv') of Theorem \ref{transmain}) is idempotent and exact; hence the class of objects
of $\au_{\le l}=W_{\le l}\hrt$  contains all subobjects and factor-objects of its elements (in $\hrt$). Thus the long exact sequences coming from applying $\tau_{=j}$ to the $\cu$-extensions  given by Lemma \ref{lbas}(5) yields the reduction in question (by induction; here we also use Remark \ref{rtf}(3)). %???!!

Lastly,
% as shown in the proof of   Theorem \ref{transmain},
by Remark \ref{rtf}(2) 
we have $\cu^{w=0}=\bigoplus_{j\in \z}\obj \au_j[-j]$; the result follows immediately.

%laja!!!

2. By the previous assertion, it suffices to verify the statement for $X\in \cu^{t=j}$. Then the fact is immediate from the statement that  'nice' filtrations of $X[j]$ yield its %weight 
nice decompositions. % (as in (\ref{ewd}).

II1. As noted in the proof of Theorem \ref{transmain}, this functor is the restriction to $\hrt$ of the functor $b_i$ that is right adjoint to the embedding $\cu_{\le i}\to \cu$. The result follows immediately.

2. Dual to the previous assertion (see Remark \ref{rtf}(3)). 

III All of these assertions are easy consequences of the existence of a weight filtration for $\hrt$ (see Definition \ref{dwfilt} below).

 The functors $Gr_i$ and $Gr_i'$ are exact as compositions of exact functors. They obviously take their values in $\aui$ are are identical on it; hence they are idempotent.  
 
 Now, we can compute $Gr_i$ using the following functorial short exact sequence
 \begin{equation}\label{egri}
 0\to W_{\le i-1}X\to W_{\le i}X\to Gr_i X\to 0; 
\end{equation}
we also consider its dual 
\begin{equation}\label{egrpi}
 0\to Gr'_i X\to W_{\ge i}X\to W_{\ge i+1}X\to 0. 
\end{equation}
 We obtain that both of  $Gr_i$ and $Gr'_i$  kill $\hrt\cap \cu^{w\ge i+1}$ and  $\hrt\cap \cu^{w\le i-1}$.
Since any object of $\hrt$ could be presented as an extension of an object of $\aui$ by that of $\hrt\cap \cu^{w\ge i+1}$ and  $\hrt\cap \cu^{w\le i-1}$, we obtain that $Gr_i\cong Gr'_i$, whereas (\ref{egri}) yields the last statement in assertion III1.

Next, (\ref{egri}) yields that $Gr_i X=0$ whenever $W_{\ge i+1}X\cong W_{\ge i}X$; (\ref{egrpi}) also yields that this is equivalent to $W_{\le i}X\cong W_{\le i-1}X$. Thus we obtain assertion III2.

\end{proof}

\begin{rema}\label{rluc}

 1. %As we have already noted in the proof, part I1 of the proposition is equivalent to: 
 So  $\tau_{\le i}$ and $\tau_{\ge i}$ preserve $\cu^{w=0}$. %Indeed??!!

%These (equivalent) statements are 
This statement is somewhat weaker than the transversality of weight and $t$-structures (in contrast to  condition (iv) of Theorem \ref{transmain} where $t$ and $w$ are 'permuted'), since %they do not 
it does not imply the semi-simplicity of the corresponding $\au_i$. For example, let $\au$ be a non-semi-simple abelian category %with enough projective objects
such that any object of $\au$ has finite projective dimension; then $\cu=D^b(\au)\cong K^b(Proj \au)$. Then we can consider the 'stupid' weight structure on $\cu$ whose heart is $Proj \au$. Certainly, $\hw$ is preserved by the truncations with respect to the canonical $t$-structure for $\cu$. Yet if we put $\au_0= Proj \au$ and all other $\aui=0$, non-projective objects of $\hrt\cong \au$ wouldn't have filtration by objects of $\au_i$ (i.e. by projective ones). 

One could  also consider the direct sum of a collection of ('shifted') examples of this sort in order to get more than one non-zero $\aui$.

2. For %$\au_{\ge l}$ being the categorical image of the functor $W_{\ge l}(-):\hrt\to \hrt$
%and $\au_{\le l-1}$ being its kernel (for all $l\in \z$) 
$\au_{\le l}$ as in the proof of assertion I1, and $\au_{\ge l+1}$ being the categorical kernel of $W_{\le l}(-):\hrt\to \hrt$
we can re-formulate assertion I2 as follows:   $X\in \cu^{w\ge 0}$ (resp. $X\in \cu^{w\le 0}$) whenever $X^{\tau=i}\in \au_{\ge i}$ (resp. $X^{\tau=i}\in \au_{\le i}$) for all $i\in \z$ (for $X\in \obj \cu$). %??? Define?? 2.  %Part I1 of the proposition is

This statement corresponds to the definition of weights for mixed Hodge complexes (by Deligne) and for complexes of mixed Hodge modules (by Saito); see \S\ref{shodge} below. For Artin-Tate motives over a number field this result was established in Theorem 3.8 of \cite{wildat}. 
\end{rema}

We are also able to prove a simple formula for the Grothendieck group of $\cu$. 

\begin{pr}\label{pkzt}
Define  $K_0(\cu)$ as a group whose generators are $[C],\ C\in \obj \cu$; if
$D\to B\to C\to D[1]$ is a distinguished triangle then we set
$[B]=[C]+[D]$.

Let $\cu$ possess transversal $t$ and $w$. Then $K_0(\cu)$ is a free abelian group with a basis %formed 
indexed by isomorphism classes of indecomposable objects of all of $\aui$ (for $i$ running through all integers).
\end{pr}
\begin{proof}
As we have already noted, $\hw$ is idempotent complete since all $\aui[-i]$ are (by Lemma \ref{lbas}(6)). By Theorem 5.3.1 of \cite{bws} we obtain that $K_0(\cu)\cong K_0(\hw)$. Recall that here the Grothendieck group  of $\hw$  is defined as a group whose generators are of the form $[X],\
X\in\obj \hw$; the relations are 
$[X]=[Y]+[Z]$ if 
 $X\cong Y\bigoplus Z$ for 
$X,Y,Z\in\obj \hw$. 

Now, obviously %the Grothendieck group of 
$K_0(\aui[-i])\cong K_0(\aui)$ is a free abelian group  
with a basis formed by isomorphism classes of indecomposable objects of $\aui$ (for any fixed $i\in \z$). 
Besides, we have a functor $\bigoplus \aui[-i]\to \hw$; it induces a surjection of $K_0$-groups since it is surjective on objects. This surjection is also injective since it is split by the map $K_0(\bigoplus_{i\in \z}\tau_{=i})$.

\end{proof}

 \section{Examples of transversality}

In this section we construct several examples of transversal weight and $t$-structures. In \S\ref{sdau} to this end we prove a certain general statement for $D^b(\au)$. In \S\ref{s1mot} we study the $1$-motivic examples of our setting (and relate them with the Chow weight structure for Voevodsky's $\dmgm$, whereas in \S\ref{shodge} we study certain Hodge-theoretic examples. Lastly, in \S\ref{sconj}
 we describe the conjectural relations between various 'structures' for $\dmgm$.

\subsection{On the main series of examples}\label{sdau}

A result from \cite{bvk} yields:  certain conditions on  $\au$ ensure  for  $D^b(\au)$ the existence of  a weight structure transversal to the canonical $t$-structure. We will use this statement for all the (main) examples in this paper.

%provides us with a large family of examples of transversal weight and $t$-structures.

\begin{pr}\label{pdab}
%If $\au$ has a weight filtration in the sense of Definition E7.2 \cite{bvk}, i.e. ,
Suppose that $1_{\au}$ is equipped with a separated exhaustive increasing system of exact
subfunctors $W_{\le i}$. For all $i\in \z$ denote the categorical kernel of the restriction of $W_{\le i-1}$ to the image of $W_{\le i}$ by $\au_i$.

Suppose that all of $\au_i$ are semi-simple. %, and $\au_i\perp \au_{j}$ for $j>i$. % for all $j\neq i$. %finite on every object???!!!
Then $\au_i$ yield a strongly semi-orthogonal generating system in $D^b(\au)$, and the corresponding weight structure is transversal to the canonical $t$-structure for $D^b(\au)$.
\end{pr}
\begin{proof}
$\{\au_i\}$ are obviously generating.
The orthogonality statements required are immediate by applying  Remark E7.8,  Lemma E7.5, and Proposition E7.4(4) of ibid.
\end{proof}

\begin{rema}\label{rfun}

1. There is a funny way to make a series of new examples of transversal weight and $t$-structures from one given example.

Suppose that a category $\cu$ with a bounded weight structure $w$ admits a {\it differential graded enhancement} 
(see Definition 6.1.2 of \cite{bws}; note that one can easily find an enhancement for $D^b(\au)$ for any small $\au$, 
since a localization of an enhanceable category is enhanceable). Then for any $N\ge 0$ there exist a triangulated category $\cu_N$ and an exact {\it truncation functor} $t_N:\cu\to \cu_N$ such that: $\cu=\lan t_N(\cu^{w=0})\ra$; and $\cu_N(t_N(A),t_N(B))$ for  $A\in \cu^{w=i}$, $B\in \cu^{w=0}$ ($i\in \z$) is zero for 
for $i>N$ and $i<0$, and is isomorphic via $t_N$ to $\cu(A,B)$ for  $0\le i\le N$ (see Remark 6.2.2 and \S6.3 of \cite{bws}).
 In particular, for $N=0$ we obtain the {\it strong weight complex} functor $t:\cu\to \cu_0\cong K^b(\hw)$; see loc.cit.

 We obtain: for a strongly semi-orthogonal generating family of  semi-simple $\{\au_i\subset \cu\}$ (and the corresponding $w$) 
 the family $\{t_N(\aui)\}$ is also strongly semi-orthogonal and generating in $\cu_N$ (since for any $i,j,s\in \z$, $X\in \obj\aui,\ Y\in \obj \au_j$ 
 the group $\cu_N(t_N(X),t_N(Y)[s])$ is either $0$, or is isomorphic to $\cu(X,Y[s])$); moreover, $t_N(\aui)$ are semi-simple. 
 
 So, using any of the examples (of transversal $t$ and $w$) described below, one obtains the existence of transversal $w$ and $t$ for all the corresponding $K^b(\hw)$ 
 and also for their 'higher' analogues (i.e the corresponding $\cu_N$ for $N>0$). In particular, the 'motivic' conjectures imply the existence a $t$-structure
  transversal to the 'stupid' weight structure (the latter is the 'simplest' weight structure with the heart $=\chow$, that 
  corresponds to the stupid truncations of complexes; see \S1.1 of \cite{bws}) for $K^b(\chow)$; this is true unconditionally for 
  the '$1$-motivic' part of this category. % (over any smooth variety $S/k$).??!! Not the same category for $S\neq \spe k$??!!
 
 The author doubts that $\cu_N$ for $\cu=D^b(\au)$ (as in the proposition) are always isomorphic to the derived categories of the
  corresponding $\hrt_N$; in any case, this construction surely produces some 'new' examples from the ones that we will describe below.

2. Below we will  consider several motivic and 'Hodge-theoretic' examples of our setting. All the motives, Hodge structures, complexes, and 
modules, and connecting functors between them that we will consider below will %(by default)
 have rational coefficients. This is because the results of this paper cannot be applied (directly) to motives %and $1$-motives 
 with integral coefficients. Indeed, even the category of finitely generated abelian groups (the 'easiest' part of motives of weight zero) 
 is not semi-simple. %Yet abelian categories of cohomological dimension $1$ are easy to assemble into a 'filtered $t$-structure'??!!!

Note also that people usually do not expect Voevodsky's motives with integral coefficients to possess a 'reasonable' motivic $t$-structure 
(see \S3.4 and Proposition 4.3.8 of \cite{1}).

\end{rema}

\subsection{\texorpdfstring{$1$}{1}-motives; their relation with Voevodsky's motives (endowed with the Chow weight structure)}\label{s1mot} 

First we consider the triangulated category $\dom$ of $1$-motives. % (over $k$).

\begin{pr}\label{pomfil}
Let $S$ be connected %??!!
and regular essentially of finite type over $k$. Then the category $\dom(S)$ (the derived category of Deligne's $1$-motives over $S$; see Appendix C.12 of \cite{bvk}) is equipped with a weight structure $w_1$ that is transversal to the canonical $t$-structure for it.

\end{pr}
\begin{proof}

By Proposition %15.2.1??
C12.1 of %\cite{bvk},
ibid., the category $\mgo=\au$ satisfies the conditions of Proposition \ref{pdab}.
\end{proof}

%\subsection{Voevodsky's motives: the Chow weight structure, and comparison functors}

Next we set $S=\spe k$ and recall that  the category $\dmge(\subset \dmgm)$ of effective geometric Voevodsky's motives over $k$ possesses a certain {\it Chow} weight structures whose heart is $\chowe$; see \S6.6 of \cite{bws}.

\begin{pr}\label{pcomp}
1. The embedding $T:\dom\to\dmge$ defined in Theorem 2.1.2 of \cite{bvk} is weight-exact (i.e. $T(\dom^{w_1\le 0})\subset \dmge^{\wchow\le 0}$ and  $T(\dom^{w_1\ge 0})\subset \dmge^{\wchow\ge 0}$).

2. The functor $\lalb:\dmge\to \dom$ introduced in Definition 5.2.1 of ibid. is weight-exact also, as well as $\rpic$ (see Definition 5.3.1 of ibid.; since $\rpic$ is contravariant, here weight-exactness means that $\rpic(\dmge^{\wchow\le 0})\subset \dom^{w_1\ge 0}$, and $\rpic(\dmge^{\wchow\ge 0})\subset \dom^{w_1\le 0}$). %Use Corollary 5.3.2: duality??!!

\end{pr}
\begin{proof}
1. Since $w_1$ is bounded, it suffices to verify that $T(\dom^{w_1=0})\subset \obj \chowe$.

To this end it suffices to prove (in the notation of Theorem \ref{transmain}) that $\aui[i]\subset \chowe$. This is immediate from the description of $\aui$, that could be immediately obtained from Definition C11.1 of \cite{bvk} along with Lemma 16.1.1 of ibid.
 %???!!!

2. It suffices to verify that $\lalb$ and $\rpic$ map Chow motives into Chow ones.

Now, (for any $X\in \obj\dmge$) $\lalb(X)$ could be obtained from $L_i\alb(X)[i]$ (see Definition 8.1.1 of ibid.) %???!! [-i]??!!
 via extensions (as usual for cohomology coming from a $t$-structure). Since the heart of a weight structure is always extension-stable, it suffices to verify that $L_i\alb[i]$ sends any smooth projective variety $P/k$ into $\chowe$. We have $L_i\alb(P)=0$ for $p\neq 0,1,2$. Moreover,  Corollary 10.2.3 of ibid. immediately implies that  $L_i\alb(P)[i]$ is a Chow motif for $i=0,2$.
 The case $i=1$ is immediate from Lemma 16.1.1 of ibid.

The result for $\rpic$ follows easily, since the functors are interchanged by Poincare %Cartier??!! '$1$-motivic'???!!!
 duality (see Corollary 5.3.2 of ibid.); note that Poincare duality maps Chow motives into Chow ones.
 Alternatively, one could apply Corollary 10.6.1 of ibid. (combined with Lemma 16.1.1 of ibid.).

\end{proof}

\begin{rema}
1. Very probably, an analogue of (at least) part 1 of the proposition  is also fulfilled for motives over any $S$ that is regular and essentially of finite type over $k$. Recall that a certain version of Voevodsky's  $S$-motives (with rational coefficients) was thoroughly studied in \cite{degcis}; a Chow weight structure for this category was introduced in \cite{hebpo} and  in \cite{brelmot}. The main difficulty here is to construct %(some) 
a comparison functor.

2. As shown in \S2 of \cite{bws}, for any (co)homology theory defined on $\cu$ a weight structure $w$ for it yields certain {\it weight filtrations} (cf. %Theorem\ref{transmain}(v,v') above)
Definition \ref{dwfil} below), {\it weight spectral sequences} %(cf. Remark \ref{rluc}(3))
(cf. Proposition \ref{pwss}(I)), and {\it virtual $t$-truncations}. The Proposition immediately yields the compatibility of all of these notions for $\dom$ and $\dmge$ (with respect to $T$, $\lalb$, and $\rpic$). Also, these comparison functors respect the weight complex functor (see \S3 of ibid. and \S6.3 of \cite{mymot}).

3. An analogue of Proposition \ref{pomfil} (along with Proposition \ref{pcomp}(1)) for Artin-Tate motives over a number field  was established in \S3 of \cite{wildat}.

\end{rema}

\subsection{Graded polarizable mixed Hodge complexes and Hodge modules; the Hodge  realization}\label{shodge}

\begin{pr}\label{phodge}
I Let $X$ be a complex variety.

1. There
 exists a weight structure $w$ on the category $\dbmhmx$ (the derived category of mixed Hodge modules over $X$; see \cite{sai}) that is transversal to the canonical $t$-structure for it. 
 
 2. For this weight structure $\dbmhmx^{w\le 0}$ is the class of complexes of mixed Hodge modules of weight $\le 0$, and $\dbmhmx^{w\ge 0}$ is the class of complexes of mixed Hodge modules of weight $\ge 0$
in the sense of Definition 1.6 of ibid. %\cite{sai}.

II Let $k$ be a subfield of the field of complex numbers.

1. There exists a weight structure $w_{Hodge}$ for the category $\dbhp(k)$ introduced in \S3 of \cite{bah}, that is transversal to the canonical $t$-structure for it (given by Lemma 3.11 of ibid.).

2. The Hodge realization functor $\hsing(k):\dmgm(k)\to\dbhp(k)$ (for example, the composition of the 'polarizable mixed realization' one constructed in \S2.3 of \cite{hu} with the natural functor 'forgetting all other realizations'; cf. \S17.2 of \cite{bvk}) is weight-exact with respect to these weight structures  (i.e. $\hsing (\dmgm(k)^{\wchow\le 0})\subset \dbhp(k)^{w_{Hodge}\ge 0}$ and  $\hsing(\dmgm(k)^{\wchow\ge 0})\subset \dbhp(k)^{w_{Hodge}\le 0}$).

%$\hsing$ strictly respects weights on $\dom$???!! Remark??!!! Ask Luca??!!

\end{pr}
\begin{proof}
I1. We should verify that  $\mhmx$ satisfies the conditions of Proposition \ref{pdab}. This is immediate from Propositions 1.5 and 1.9 of \cite{sai}. 

2. Immediate from Definition 1.6 of \cite{sai} along with Remark \ref{rluc}(2).

II1. By Lemma 3.11 of \cite{bah}, $\dbhp(k)$ is isomorphic to the bounded derived category of the abelian category of  graded polarizable mixed Hodge structures (over $k$). Similarly to the proof of assertion I1, it remains to apply Proposition \ref{pdab}. %more??!!

2. As in the proof of Proposition \ref{pcomp}, we should verify that $\hsing$ maps the heart of $\wchow(k)$ into that of $w_{Hodge}(k)$. To this end it suffices to note that the $i$-th cohomology of a smooth projective $P/\mathbb{C}$ is a pure (polarizable) Hodge structure of weight $i$ %$-i$??!!!
(for $i\in \z$).
\end{proof}

\begin{rema} 
1. The weight filtration on $\dbhp$ obtained corresponds to the Deligne's definition of weights for mixed Hodge complexes (cf. assertion I2).
%signs of weights??!!

2. For a certain category of {\it Beilinson motives} over $X$ (these are relative Voevodsky's motives
with rational coefficients over $X$; see \cite{degcis}) a certain Chow weight structure
was introduced in  \cite{hebpo} and \cite{brelmot}).
Very probably, this weight structure is   compatible (similarly to assertion II2) with the one introduced in assertion I1 (and so with Saito's weights for complexes
 of mixed Hodge modules); note also that the 'functoriality' properties of the weights of the latter (see Proposition 1.7 of \cite{sai}) are parallel 
to those for $X$-motives (as described in Theorem 3.7 \cite{hebpo} and Theorem 2.2.1 in  \cite{brelmot}).
The problem is that (to the knowledge of the author) no 'mixed Hodge module'  realization of $X$-motives is known to exist at the moment.

3. It does not seem very difficult to extend assertion II to Huber's category $\dmrp$ of mixed realizations (introduced in \S21 of \cite{hubook}), the corresponding cohomology functor (see the Remark after Corollary 2.3.4 of \cite{hu}), and also to its \'etale cohomology analogue.

On the other hand, for a finite characteristic $k$ one cannot prove that the \'etale cohomology of smooth projective varieties is polarizable in absence of the Hodge standard conjecture (for this case). Hence one could only define realizations with values in categories that do not satisfy any polarizability conditions; whereas without these restrictions there will exist non-trivial $1$-extensions inside $\aui$ for a single $i$ (and so, $\aui[-i]$ will not be contained in the heart of any weight structure).

\end{rema}

\subsection{Mixed motives:  the conjectural picture}\label{sconj}

It is widely believed that Voevodsky's $\dmgm$ (over a base scheme $S$) possesses the so-called motivic $t$-structure, whose heart $MM(S)$ (the category of {\it mixed motivic sheaves} over $S$) possesses a weight filtration (cf. Definition \ref{dwfilt} below) whose 'pure' factors $\aui$ are abelian semi-simple (see \S5.10A in \cite{beilh}). Moreover, at least in the case when $S$ is the spectrum of a (perfect) field $k$, people believe that all of $\aui[i]$ consist of Chow motives. This conjectures yield (immediately by %condition (iii) of 
Theorem \ref{transmain}) that the Chow weight structure for $\dmgm(S)$ (that is known to exist unconditionally)
is transversal to the motivic $t$-structure. Thus,  that the results of section 1 could be applied to this situation. We study these questions in detail and obtain some new (conjectural) information on motives in \cite{bmm} this way. Note that 'classically' people were looking only for the mixed motivic $t$-structure and for the 'weight filtration' for $\dmgm$; cf. the introduction and Remark \ref{rwfilc} below. The notion of weight structure is new in this picture; it allows us to construct certain 'weights' for motives unconditionally.

Now we show that the notion of   transversal $t-$ and weight structures indeed yields a way to relate the 'classical' approach to weights with 'our' (i.e. the 'Chow weight structure') one. For  $X\in \dmgm^{t=0}$, a nice % weight 
decomposition
%choice of (\ref{ewd}) %
of $X$
 (see condition (iv) of Theorem \ref{transmain})  yields a distinguished triangle of the type (\ref{ewf}) (see condition (iii') of the Theorem). Now, for any (contravariant) cohomology theory $H:\dmgm\to \au$  the 
 $-1-j$-th level of the 'weight filtration' of $H^i(X)=H(X[-i])$ ($X\in \obj \dmgm,\ i,j\in \z$)  is 'classically' defined as $H^i(X_2)$, where $X_2$ is taken from (\ref{ewf}). Now suppose that
 $H$  factorizes through $MM$ (i.e. it is the composition of the functor $M\mapsto M^{\tau_{MM}=0}$ with a contravariant exact functor; note that conjecturally one could factorize through $MM$ all cohomology theories endowed with 'classical' weights); 
 then $H^i(X)$ and its weight filtration coincides with those for $X'=\tau_{=-i}X$ %(the $-j$-th cohomology of $X$ with respect to $t$ shifted by $[j]$;
 (here we use the fact that the weight filtration functors given by Lemma \ref{lwf}(II) commute with $t$-truncations).
 Hence
 $$H^i(X_2)\cong H^i(X_2')\cong \imm (H^i(w_{\ge j+1-i}X')\to H^i(X'))\cong \imm (H^i(w_{\ge j+1-i}X)\to H^i(X)).$$ Now note that the last two terms do not depend on the choices of the corresponding weight decompositions (see Proposition 2.1.2 of \cite{bws} or Definition \ref{dwfil} below). Hence one can define the weight filtration of $H^j(X)$ unconditionally (using the last term of the formula)!

 One can also make a similar observation using a 'weight' filtration on $\au$ (if it exists); see Remark 2.4.3 of \cite{bws}.
 
Whereas the approach for defining weights for cohomology described (above) is somewhat 'cheating' when it we apply it to  pure motives (since it usually gives no new information); yet it yields interesting results in the 'mixed' case. Note here: for $X=\mg(P)$, where $P$ is smooth projective over $k$ (or over the corresponding base $S$), the motives $X_1,X_2$ should come from a {\it Chow-Kunneth decomposition} of $X$; so their construction is completely out of reach at the moment (in general).

\begin{rema}\label{rwfilc}

One can take the setting of  Lemma \ref{lwf}(II3)
for the definition of a weight filtration for a triangulated category; see also the equivalent
 Definition E17.1 of \cite{bvk} 
(this is a 'triangulated analogue' of Definition \ref{dwfilt} below). 

Now we describe how one could obtain  a weight filtration for $\dmgm$ assuming that the motivic $t$-structure exists. 
If one has transversal $t,w$ for $\cu$, then using the corresponding $\au_i$ one can define the triangulated categories $\cu_i=\lan \au_i\ra$. Next one can use Lemma \ref{lwf}(I) and introduce the corresponding categories  $\cu_{\le i}$. % and the %connecting  projection functors  $b_i:\cu\to \cu_{\le i}$; hence $\cu_{\le i}$ yield a weight filtration  for $\cu$ indeed.

It is easily seen that in the case we consider in this paragraph (i.e. for $\cu=\dmgm$, the  corresponding 'structures' and $\aui$) we should obtain the long-searched-for weight filtration for $\dmgm$ (that was already mentioned in the introduction).
\end{rema}

\section{On weight spectral sequences and weight filtrations for %the heart of 
the heart of \texorpdfstring{$t$}{t}}\label{sdeg}
%\texorpdfstring{${\underline{Ht}}$}{Ht}}\label{sdeg}

%Now 
In this section we study the relations between $w,t$, weight filtrations, and weight spectral sequences. In particular, we study  a condition on $w$ and $t$ that is strictly weaker than their transversality (yet it can be easier to verify). 

In \S\ref{sdwss} we recall weight filtrations and weight spectral sequences (as introduced in \cite{bws}); and relate the degeneration of the latter with 'exactness' of the corresponding filtrations for homology.
In \S\ref{swfilhrt} we study the when a homology theory comes from a $t$-structure; we obtain a certain weight filtration for $\hrt$ in this case.

\subsection{On weight filtrations and (degenerating) weight spectral sequences for (co)homology }\label{sdwss}

%To this end 
First we recall weight  filtrations and weight spectral sequences (as introduced in \cite{bws}).

Let $\au$ be an abelian category. In \S2 of %.3--2.4 of %\cite{bws}
ibid. for $H:\cu\to \au$ that is either cohomological or homological (i.e. it is either covariant or contravariant, and converts distinguished triangles into long exact sequences), certain {\it weight filtrations} and {\it weight spectral sequences}  (corresponding to $w$) were introduced. Below we will be more interested 
the homological functor case; certainly, one can pass to cohomology (that seems to be somewhat more actual when $\cu$ is some category of Voevodsky's motives) by a simple reversion of arrows (cf. \S2.4 of ibid.).

\begin{defi}\label{dwfil}
Let $H:\cu\to \au$ be a homological functor, $i\in \z$.

1. We denote $H\circ [i]:\cu\to \au$ by $H_i$.

2. We choose some $w_{\le i}X$ and define  the {\it weight filtration} for $H$  by $W_iH:X\mapsto \imm (H(w_{\le i}X)\to H(X))$.

Recall that $W_iH$ is functorial in $X$ (in particular, it does not depend on the choice of  $w_{\le i}X$); see Proposition 2.1.2(1) of ibid. (this fact  also easily follows from Lemma \ref{lbas}(2)).
\end{defi}

%$D_1$!!!!
Now we recall some of the properties of weight spectral sequences, and prove some easy (new) results in the case when they degenerate.

\begin{pr}\label{pwss} Let $H:\cu\to \au$ be a homological functor.

I For any $X\in \obj \cu$  there exists a spectral sequence $T=T_w(H,X)$ with $E_1^{pq}(T)=
H_q(X^{p})$ for certain $X^i\in \cu^{w=0}$ (coming from certain weight decompositions as in (\ref{ewd})) that converges to 
$E_{\infty}^{p+q}=H_{p+q}(X)$.
 $T$ is   %covariantly
$\cu$-functorial  in $X$ and in $H$ (with respect to  composition of $H$ with exact functors of abelian categories) starting from $E_2$.  
Besides, the step of filtration given by ($E_{\infty}^{l,m-l}:$ $l\ge k$)
 on $H_{m}(X)$ (for some $k,m\in \z$) equals  $(W_{-k}H_{m})(X)$.
Moreover, $T(H,X)$ comes from an exact couple with $D_1^{pq}=H_{p+q}(w_{\le -p}X)$ (here one can fix any choice of $w_{\le -p}X$). %  $E_1^{pq}=H_q(X^p)$ for some $X^p\in \cu^{w=0}$ (that could be expressed in terms of $w_{\ge p}X$).
 %$T$ converges to $H_{p+q}(X)$ (since $w$ is bounded), and the level of filtration corresponding to $E_1^{pq}$ is given by $W_pH_{p+q}(X)$. $T$ is (canonical and) functorial in $X$ starting from $E_2$; it is also functorial in $H$ (starting from $E_1$ if one fixes the choice of the corresponding weight decompositions).

We will say that $T$ {\it degenerates at $E_2$} (for a fixed $H$) if $T_w(H,X)$ does so for any $X\in \obj \cu$.

II Suppose that $T$ degenerates at $E_2$ (as above), $i\in \z$. Then the following statements are fulfilled.

1. The functors $W_iH$ and $W_i'H:X\mapsto H(X)/W_iH(X)$ are homological.

2. For any $f\in \cu(X,Y)$ the morphism $H(f)$ is strictly compatible with the filtration of $H$ by $W_i$ i.e. $W_{i}H(X)$ surjects onto $W_{i}H(Y)\cap \imm H(f)$.   %for any $i\in \z$?? 

3. Let $\bu$ be an abelian category; $F:\au\to \bu$ be an exact functor. 
Then $T(F\circ H,-)$ degenerates also. 

III Conversely,  if $F$ (as in assertion II3) is conservative and   $T(F\circ H,-)$ degenerates, then $T$ degenerates also.

\end{pr}

\begin{proof}
I All of the results stated were verified in  Theorem 2.3.2 of ibid. (see formula (13) for a precise description of the corresponding filtration), expect the fact that the functoriality of $T$ with respect to $H$ does not depend on the choice of the corresponding weight decompositions. The latter assertion is immediate from  
Theorem 2.4.2(II) of \cite{bger}. % (whereas in .

II1. The degeneration at $E_2$ yields that $(W_{-p}H_{p+q})X\cong D_2^{pq}T(H,X)$ for any $X\in \obj \cu$ (so, here we consider the derived exact couple). 
Now, we note that the latter is homological (this is the corresponding {\it virtual $t$-truncation} of $H$; see %\S2.5 
Proposition 2.5(II1) of %ibid.).
\cite{bws}). We deduce that $W_i'H$ is homological also by applying part II2 of loc.cit. %(also). 

%In order to study $W_i'H$
%Lastly,
 %we note that $W_i'H(X)\cong \imm (H(X)\to H(w_{\le i-1}X))$ (see (\ref{ewd}). Hence the definition of $W_i'H$ is dual to that of $W_{i+1}H$ (see . self-duality of wss??!!

2. We complete $f$ to a distinguished triangle $X\stackrel{f}{\to} Y \stackrel{g}{\to} Z$. Then $W_{i}H(Y)\cap \imm H(f)=W_{i}H(Y)\cap \ke H(g)=\ke (W_{i}H(Y)\to W_{i}H(Z))$. It remains to note that the last term coincides with $\imm W_{i}H(X)$, since $W_iH$ is homological.

3. Obvious. 

III Easy; note that conservative exact functors of abelian categories do not kill non-zero morphisms.

\end{proof}

\subsection{On the weight filtration for %\protect{\texorpdfstring{$\underline{Ht}$}{hrt}}}
the heart of t} %\texorpdfstring{$t$}{t} 
\label{swfilhrt}

Now we introduce the notion of a weight filtration for an abelian category; (cf. Definition E7.2 of \cite{bvk}).  

\begin{defi}\label{dwfilt}
For an abelian $\au$, we will say that an increasing family of full subcategories $\au_{\le i}\subset \au$, $i\in \z$, yield a {\it weight filtration} for $\aui$ if, $\cap_{i\in \z}\au_{\le i}=\ns$, $\cup_{i\in \z}\au_{\le i}=\au$, and there exist exact right %??!!
adjoints $W_{\le i}$ to the embeddings  $\au_{\le i}\subset \au$.
\end{defi}

\begin{lem}\label{lwfil}
Let $i$ run through all integers.
 
1. Let  $\au_{\le i}$ yield a weight filtration for $\au$. Then they are are exact abelian subcategories of $\au$. Moreover, the adjunctions yield functorial embeddings of $W_{\le i}X\to X$ such that $W_{\le i}X\subset W_{\le i-1}X$ for all $i\in \z$, and
the functors $W_{\ge i}:X\mapsto X/W_{\le i}X$ are exact also. Besides, the
 categories $\aui$ being the 'kernels' of the restriction of $W_{\le i-1}$ to $\au_{\le i}$, are abelian, and $\aui\perp \au_j$ for any $j\neq i$. %??!!

2. Let $W_{\le i}$ be a increasing sequence of subfunctors of $1_{\hrt}$. Then taking $\au_{\le i}$ whose objects are $\{X\in \obj \au: W_{\le i}(X)\cong X\}$ we obtain a weight filtration corresponding to these $W_{\le i}$. %such that these $W_i$ are the corresponding adjoint functors. % (mentioned in the definition of weight filtrations).
Moreover, all $W_{\le i}$ are idempotent functors. %draw a diagram??!!

3. If $W_{\le i}$ yield filtration for $\au$, then the functors $W_{\le -i}^{op}:\au^{op}\to \au^{op}$ yield a weight filtration also.

%4. The morphisms in $\au$ strictly respect the weight filtration i.e. for $f\in \au(X,Y)$, $X,Y\in \obj \au$, and any $i\in \z$, we have: $W_{\ge i}X$ surjects onto $W_{\ge i}Y\cap \imm f$.   
\end{lem}
\begin{proof}
Assertions 1,2 are immediate from Proposition E7.4 %(2)
 and Remark E7.8   of \cite{bvk}.

3: immediate from assertion 2.

\end{proof}

Now we fix certain (bounded) $w$ and $t$ for $\cu$, and study a  condition  ensuring that $w$ induces a weight filtration for $\hrt$.

\begin{pr}\label{pdeg}
Let $H=(-)^{\tau=0}$; let $i$ run through all integral numbers.

I Suppose that the corresponding $T$ degenerates. Then the following statements are fulfilled.

1.  The functors $W_iH:\cu\to \hrt$ are homological. %; they kill $\cu^{t\le 0}$ and $\cu^{t\ge 1}$.
 The restrictions $W_{\le i}$ of $W_iH$ to $\hrt$ define a weight filtration for this category. Besides, $W_iH\cong W_{\le i}\circ H$.

2. For $X\in \cu^{t=0}$ %$i\in \z$ %??
 we have: $X\in \obj \au_{ i}$ whenever there exists a bounded complex $C=\dots \to C^{-1}\to C^0\to C^{1}\to\dots$ in $\hw$ %(for all $j\in \z$)
  and $X$ is isomorphic to the $0$-th cohomology of the complex $C^{\tau=i}$, whereas for any $l,j\in \z$ we have: 
 the  $l$-th cohomology of $C^{\tau=j}$ is zero unless $l=0$, $j=i$.

3. For $X\in \cu^{t=0}$, %$i\in \z$ %??
 we have: $X\in \obj \au_{\le i}$ whenever there exist a $Y\in \cu^{w\le i}$ and an epimorphism $H(Y)\to X$.

%2. Let $t$ be transversal to $w$. Then $T$ degenerates.

II Suppose that $t$ is transversal to $w$. %$\tau_{=i}$ preserves $\hw$ for any $i\in \z$. 
Then $T$ degenerates. The corresponding $\aui$,  $\au_{\ge i}$, $W_{\ge i}$, and $W_{\le i}$ are the same as those described in \S\ref{sdtrans}. %, whereas $W_i'=W_{\le i-1}$.

III  Let $\bu$ be an abelian category; let $F:\hrt\to \bu$ be an exact functor.

1. Suppose that  $T$ degenerates. Then   $T_w(F\circ H,-)$ also does.

2. Conversely,  suppose that $F$ is conservative and $T_w(F\circ H,-)$ degenerates.  Then $T$ degenerates. 

Moreover, for $X\in \cu^{t=0}$ we have: $W_{\le i}X=X$ (resp. $W_{\le i}X=0$) whenever $W_i(F\circ H)(X)=F(X)$ (resp. $W_i(F\circ H)(X)=0$).

\end{pr}
\begin{proof}

I1. Immediate from Proposition \ref{pwss}(II1) and Lemma \ref{lwfil}(2).

%???!!!

2. If $X\in \aui$, then (by definition of $\aui$) for the corresponding weight spectral sequences we obtain: $W_{\le i}X=X$, $W_{\le i-1}X=0$. %As was shown %just above,
%in the proof of Proposition \ref{pwss}(II1), 
This translates into (see the proof of Proposition \ref{pwss}(II1)): $D_2^{-i,i}X\cong X$ and $D_2^{1-i,i-1}X=0$. Hence $X\cong E_2^{-i,i}$, whereas all the remaining $E_2^{pq}$ are zero (note that $E_2^{p+q}(T)=0$ for any $X\in \cu^{t=0}$, $p+q\neq 0$). Therefore, $X$ is the $-i$-th cohomology of the complex $(C^j)=(E_1^{j,i})=(H_{i}(X^{j}))$, whereas all the other cohomology of this complex is zero, as well as the all of the cohomology of the complexes
$(X^{j})^{\tau=l}$ for all (fixed) $l\neq 0$.

Conversely, for any $C^0\in \cu^{w=0}$ we obviously have $\tau^{=i}C^0\in \aui$; hence this is also true for any subfactor of $H(C^0[i])$. Here we only use the fact that  $\aui$ is an exact subcategory of $\au$; no other restrictions on the corresponding complex $C$ are necessary.

3. By definition, $X\in \au_{\le i}$ whenever $W_iH(X)=X$. Hence for $X\in \au_{\le i}$ we can take $Y=w_{\le i}X$.

Conversely, let there exist an epimorphism $H(Y)\to X$ for $Y\in \cu^{w\le i}$.  Since $\au_{\le i}$ is an exact subcategory of $\au$, we may assume that $X=H(Y)$. Hence $W_iH(Y)=H(Y)$. Now,  assertion I1 yields that
$W_iH(Y)=W_{\le i}(H(Y))$; the result follows immediately.

%???!!!

II By the previous assertion, we have $E_1^{pq}T\in \au_{-p-q}$ %q,-q,-p??!
for any $p,q\in \z$; hence the same is true for $E_r^{pq}$ for any $r\ge 1$. Hence the %? 
boundary morphisms of $E_i(T)$ for $i\ge 2$ vanish, since their sources and targets necessarily belong to distinct $\aui$, and we obtain that $T$ degenerates at $E_2$ indeed.
 
Now, assertions I2 and I3 easily yield that $\aui$ and $\au_{\le i}$ for the corresponding weight filtration of $\hrt$ are the same as described in \S\ref{sdtrans}
 (for transversal $w$ and $t$). Hence the two versions of $W_{\le i}$ and $W_{\ge i}$ coincide also.

III Assertion 1 is  just a partial cases of Proposition \ref{pwss}(II3). 
Besides, the first part of assertion 2  is also a partial case of part III of ibid.

In order to verify the second part of assertion III2 it suffices to note that exact functors respect weight filtrations for homology, %obvious??! {pwss}??!
whereas conservative exact functors cannot kill non-zero levels of these filtrations and non-zero $\hrt$-morphisms.

\end{proof}

\begin{rema}\label{rngen} %strict compatibility of the weight filtration with morphisms??
%In particular, assertion 2 yields that $T$ degenerates if $t$ is transversal to $w$. Moreover, 
1. $T$ also degenerates for the example described in Remark \ref{rluc}(1) (i.e. $\cu=D^b(\au)\cong K^b(Proj \au)$). Hence, the degeneration of $T_w((-)^{\tau=0},-)$ is strictly weaker than the transversality of $t$ with $w$.

%If $\hw\subset \hrt$ $t$ respects $w$??!!

Unfortunately, the author does not know of any 'description' of $w$ and $t$ in terms of 'generators' (as in Remark \ref{rtf}(2)) 
%(cf. the description of $t$ and $w$ in terms of $\aui$ in )  
in this more general situation. %{ewf}

It would also be interesting to understand whether the degeneration of $T$ implies that $\tau_{=i}$ preserves $\hw$ (as in the example mentioned), and whether the converse implication is valid. 

2. The author does not have a lot of examples of this situation (with $w$ not transversal to $t$). The  advantage of this weaker condition is that it could be checked 'at $t$-exact conservative realizations' of $\cu$. In particular, conjecturally it is sufficient to verify the degeneration of the (Chow)-weight spectral sequences for the ('perverse') \'etale (co)homology of (Voevodsky's) motives (instead of the 'mixed motivic homology' that corresponds to the conjectural motivic $t$-structure); cf. \S3.3 of \cite{bmm}.

\end{rema}
%\begin{verbatim} 

%\end{verbatim} 
\end{document}